\documentclass[11pt]{article}
\usepackage{amsmath}
\usepackage{amssymb}
\usepackage{amsthm}
\usepackage{cite}
\usepackage{pgf,tikz,pgfplots}
\pgfplotsset{compat=1.14}
\usepackage{mathrsfs}
\usetikzlibrary{arrows}
\usepackage{mathrsfs}
\usepackage{enumitem}

\newtheorem{theorem}{Theorem}

\newtheorem{lemma}[theorem]{Lemma}
\newtheorem{proposition}[theorem]{Proposition}

\theoremstyle{definition}
\newtheorem{note}[theorem]{Note}
\newtheorem{definition}[theorem]{Definition}
\newtheorem{question}[theorem]{Question}
\newtheorem{example}[theorem]{Example}
\newtheorem{notation}[theorem]{Notation}

\newtheorem{thmx}{Theorem}

\renewcommand{\b}{\leq_B}           
\newcommand{\R}{\mathbb{R}}
\newcommand{\N}{\mathbb{N}}
\newcommand{\Z}{\mathbb{Z}}
\newcommand{\Acc}{\text{Acc}}
\newcommand{\Cl}{\text{Cl}}
\newcommand{\Fix}{\text{Fix}}
\newcommand{\Inj}{\text{Inj}}

\def\subjclass#1{\par\medskip
\noindent\textbf{Mathematics Subject Classification (2010):} #1}
\def\keywords#1{\par\medskip
\noindent\textbf{Keywords.} #1}

\begin{document}
\makeatletter

\title{Classification of one dimensional dynamical systems by countable structures}

\author{Henk Bruin\\ \\
Faculty of Mathematics \\
University of Vienna \\
Oskar Morgensternplatz 1 \\
1090 Vienna, Austria
\\
\texttt{henk.bruin@univie.ac.at}
\and 
Benjamin Vejnar\footnote{This work has been supported by Charles University Research Centre program No.UNCE/SCI/022.}
\\ \\
Faculty of Mathematics and Physics \\
Charles University \\
Prague, Czechia\\ \\
\texttt{vejnar@karlin.mff.cuni.cz}}

\maketitle

\begin{abstract}
We study the complexity of the classification problem of conjugacy on dynamical systems on some compact metrizable spaces. 
Especially we prove that the conjugacy equivalence relation of interval dynamical systems is Borel bireducible to 
isomorphism  equivalence relation of countable graphs. This solves a special case of the Hjorth's conjecture which states that every orbit equivalence relation induced by a continuous action of the group of all homeomorphisms of the closed unit interval is classifiable by countable structures.
We also prove that conjugacy equivalence relation of Hilbert cube homeomorphisms is Borel bireducible to the universal orbit equivalence relation.
\end{abstract}



\section{Introduction}
 Measuring the complexity of relations on structures is a very general task. 
 In this paper we use the notion of \emph{Borel reducibility} (see Definition~\ref{def:Borel_red})
 and the results of invariant descriptive set theory to compare the complexities of classification problems. For more details on invariant descriptive set theory we refer to the book by Gao \cite{Gao}. For a short and nice introduction to the theory of Borel reductions we refer to a paper by Foreman \cite{ForemanIntroduction}.

 Several equivalence relations  became milestones in this theory. Let us mention four of those, which describe an increasing chain of complexities:
\begin{itemize}[noitemsep]
\item the equality on an uncountable Polish space,
\item the equality of countable sets of real numbers,
\item the $S_\infty$-universal orbit equivalence relation ($S_\infty$ is the group of permutations on $\N$), 
\item the universal orbit equivalence relation.
\end{itemize}

Let us give several \emph{examples} to make the reader more familiar with the above relations.
A classical example is a result of Gromov (see e.g. \cite[Theorem~14.2.1]{Gao}) who proved that the isometry equivalence relation of compact metric spaces is a \emph{smooth} equivalence relation, which means that it is Borel reducible to the equality of real numbers (or equivalently of an uncountable Polish space).
The isomorphism relation of countable graphs or the isomorphism relation of countable linear orders are Borel bireducible to the $S_\infty$-universal orbit equivalence.
The homeomorphism equivalence relation of compact metrizable spaces and the isometry relation of separable complete metric spaces were proved by Zielinski in \cite{Zielinski} and by Melleray in \cite{Melleray}, respectively, to be Borel bireducible to the universal orbit equivalence relation 
(see the survey paper by Motto Ros~\cite{MottoRos}). 

In order to capture all the structures in one space we need some sort of \emph{coding}. This can be done by considering some universal space (e.g. the Hilbert cube or the Urysohn space)
and all its subspaces with some natural Polish topology or Borel structure (e.g the hyperspace topology or the Effros Borel structure).
Sometimes there are other natural ways to encode a given structure. For example the class of separable complete metric spaces can be coded by the set of all metrics on $\N$ where two metrics are defined to be equivalent if the completions of the respective spaces are isometric. Fortunately in this case, by \cite[Theorem 14.1.3]{Gao} it does not matter which of the two coding we choose.
It is generally believed that this independence on a natural coding is common to other structures and thus the statements are usually formulated for all structures without mentioning the current coding. Nevertheless, for the formal treatment some coding is always necessary. 

The aim of this paper is to determine the complexity of some classification problems of dynamical systems up to conjugacy. Dynamical systems of a fixed compact metrizable space $X$ can be naturally coded as a space of continuous mappings of $X$ into itself, with the uniform topology. This one as well as the subspace of all self-homeomorphisms is well known to be a Polish space.

Let us mention several results which are dealing with the complexity of conjugacy equivalence relation.
It was proved by Hjorth that conjugacy equivalence relation of homeomorphisms of [0,1] is classifiable by countable structures \cite[Section 4.2]{Hjorth} (in fact Borel bireducible to the universal $S_\infty$-orbit equivalence relation) but conjugacy of homeomorphisms of $[0,1]^2$ is not \cite[Section 4.3]{Hjorth}.
By a result of Camerlo and Gao, conjugacy equivalence relation of both selfmaps and homeomorphisms of the Cantor set are Borel bireducible to the $S_\infty$-universal orbit equivalence relation  \cite[Theorem~5]{CamerloGao}.
Kaya proved that conjugacy of pointed minimal Cantor dynamical systems is Borel bireducible to the equality of countable subsets of reals \cite{Kaya}. 
Conjugacy of odometers is smooth due to Buescu and Stewart \cite{BuescuStewart}.
The complexity of conjugacy of Toepliz subshifts was treated several times -- by Thomas, Sabok and Tsankov, 
and by Kaya \cite{Thomas, SabokTsankov, KayaSubshift}.
Conjugacy of two-sided subshifts is Borel bireducible to the universal countable Borel equivalence relation
due to Clemens \cite{Clemens}. There is an extensive exposition of results on the complexity of conjugacy equivalence relation on subshifts of $2^G$ for a countable group $G$ in the book by Gao, Jackson and Seward \cite[Chapter 9]{GaoJacksonSeward}.
Recently, during the 8th Visegrad Conference on Dynamical Systems in 2019 it was announced by Dominik Kwietniak that conjugacy of shifts with specification is Borel bireducible to the universal countable Borel equivalence relation.

In this paper, we deal with some of the missing parts.
By mainly elementary and standard tools (excluding the complexity level of countable structures), we prove:

\begin{thmx}[see Theorem \ref{bireducible2}]
The conjugacy equivalence relation of interval maps is Borel bireducible to the $S_\infty$-universal orbit equivalence relation. 
\end{thmx}

Also, we prove:

\begin{thmx}[see Theorem \ref{Hilbert}]
The conjugacy equivalence relation of homeomorphisms as well as conjugacy of selfmaps of the Hilbert cube is Borel bireducible to the universal orbit equivalence relation.
\end{thmx}

 To this end we use some tools of infinite dimensional topology and a result of Zielinski on the complexity of homeomorphism equivalence relation of metrizable compacta \cite{Zielinski} combining with some ideas of P. Krupski and the second author \cite{KrupskiVejnar}. 
Finally we make a small overview on the complexity of conjugacy equivalence relation of dynamical systems on the Cantor set, on the interval, on the circle and on the Hilbert cube.

\section{Definitions and notations}

Let us define some standard notions from descriptive set theory (see e.g. \cite{Kechris}). A \emph{Polish space} is a separable completely metrizable topological space.
Recall that a  \emph{standard Borel space} is a measurable space $(X, \mathcal S)$ such that there is a 
Polish topology $\tau$ on $X$ for which the family of Borel subsets of $(X,\tau)$ is equal to $\mathcal S$. In order to compare the complexities of equivalence relations we use the notion of Borel reducibility.

\begin{definition}\label{def:Borel_red}
Suppose that $X$ and $Y$ are sets and let $E$, $F$ be equivalence relations on $X$ and $Y$ respectively. 
We say that $E$ is \emph{reducible} to $F$, and we denote this by $E\leq F$, if there exists a mapping $f\colon X\to Y$ such that
\[x E x' \iff f(x) F f(x'),\]
for every $x, x'\in X$.
The mapping $f$ is called a \emph{reduction} of $E$ into $F$.
If the sets $X$ and $Y$ are endowed with Polish topologies (or standard Borel structures), we say that $E$ is \emph{Borel reducible} to $F$, and we write $E\leq_B F$, if there is a reduction $f\colon X\to Y$ of $E$ into $F$ which is Borel measurable.
We say that $E$ is \emph{Borel bireducible} to $F$, and we write $E \sim_B F$, if $E$ is Borel reducible to $F$ and $F$ is Borel reducible to $E$.
\end{definition} 
In a similar fashion we  define being \emph{continuously reducible} if in addition $X$ and $Y$ are Polish spaces and $f$ is continuous.

In the whole paper we set $\mathbb N$ for all positive integers, $I=[0,1]$ and we denote the closure operator by $\Cl$.
For a separable metric space $X$ we denote by $\mathcal K(X)$ the hyperspace of all compacta in $X$ with the 
Hausdorff distance $d_H$ and the corresponding Vietoris topology. 
If $X$ is a Polish space $\mathcal K(X)$ is known to be Polish.
For compact metric spaces $X, Y$ we consider the space $C(X, Y)$ of all continuous mappings of $X$ into $Y$ with the supremum metric. In this way we get a Polish space. Especially the collection of all continuous selfmaps of $X$ is denoted shortly by $C(X)$. We also denote by $\Inj(X, Y)$ the collection of all embeddings of $X$ into $Y$ and by $\mathcal H(X)$ the collection of all homeomorphisms of $X$. These are again known to be Polish spaces.


The equality equivalence relation of real numbers is denoted as $E_=$. We denote by $E_{=^+}$ the equivalence relation on $\R^{\N}$ defined by $(a_n)E_{=^+} (b_n)$ if and only if $\{a_n\colon n\in\N\}=\{b_n\colon n\in\N\}$. The last equivalence relation is called the \emph{equality of countable sets}.

We say that an equivalence relation $E$ defined on a standard Borel space $X$ is \emph{classifiable by countable structures} if there is a countable relational language $\mathcal L$ such that $E$ is Borel reducible to the isomorphism relation of $\mathcal L$-structures whose underlying set is $\N$.
An equivalence relation $E$ on a standard Borel space $X$ is said to be an \emph{orbit equivalence relation} if there is a Borel action of a Polish group $G$ on $X$ such that $xEx'$ if and only if  there is some $g\in G$ for which $gx=x'$.

Let $\mathcal C$ be a class of equivalence relations on standard Borel spaces. An element $E\in\mathcal C$ is called \emph{universal} for $\mathcal C$ if $F\leq_B E$ for every $F\in\mathcal C$.
It is known that for every Polish group $G$ there is an equivalence relation (denoted by $E_G$) on a standard Borel space that is universal for all orbit equivalence relations given by continuous $G$-actions.
We are particularly interested in the \emph{universal $S_\infty$-equivalence relation} $E_{S_\infty}$, where $S_\infty$ is the group of all permutations of $\N$.
It is known that an equivalence relation is classifiable by countable structures if and only if it is Borel reducible to $E_{S_\infty}$. Moreover $E_{S_\infty}$ is known to be Borel bireducible to isomorphism equivalence relation of countable graphs.
Also there exists a \emph{universal orbit equivalence relation} which is denoted by $E_{G_\infty}$. We should also note that all the mentioned equivalence relations are analytic sets, i.e. images of standard Borel spaces with respect to a Borel measurable map.
We have a chain of complexities
\[E_=\b E_{=^+}\b E_{S_\infty}\b E_{G_\infty}\]
and it is known that none of these Borel reductions can be reversed.

\section{Interval dynamical systems}  

 In this section we prove that conjugacy of interval dynamical systems is classifiable by countable structures.
The strategy of our proof 
is as follows. In the first part we describe a natural reduction of interval dynamical systems to some kind of countable structures.
We assign to every $f\in C(I)$ a countable invariant set $C_f\subseteq I$ of some dynamically exceptional points for $f$. Since the set $C_f$ does not need to be dense in $I$ we do not have enough information to capture the dynamics of $f$ by restricting to $C_f$. On the other hand the dynamics on the maximal open intervals of $I\setminus C_f$ is quite simple. Hence it will be enough to define an invariant countable dense subset $D_f$ in $I\setminus\Cl(C_f)$ arbitrarily. Consequently, we get that for $f$ conjugate to $g$ there exists a conjugacy of $f$ to $g$ which sends the set $C_f\cup D_f$ onto $C_g\cup D_g$.
Finally it is enough to assign to every $f\in C(I)$ a countable structure $\Psi(f)$ whose underlying set is $C_f\cup D_f$ and which is equipped with one binary relation $\leq\restriction_{C_f\cup D_f}$ and one mapping $f\restriction_{C_f\cup D_f}$ (which can be as usual considered as a binary relation). 
We will prove then that if two such structures $\Psi(f)$ and $\Psi(g)$ are isomorphic then $f$ and $g$ are conjugate.

In the second part we prove that this reduction can be modified using some sort of coding so that the assigned countable structures share the same support and so that the new reduction is Borel. To this end we use Lusin-Novikov selection theorem \cite[Theorem 18.10]{Kechris} several times.

For $g\in C(I)$ we denote by $\Fix(g)$ the set of fixed points of $g$, i.e. those points for which $g(x)=x$.
We omit the proof of the following ``folklore'' lemma. The key idea of the proof is the back and forth argument.

\begin{lemma}\label{folk}
Let $f, g\in C(I)$ be increasing homeomorphisms such that $\Fix(f)=\Fix(g)=\{0,1\}$ and let $A, B\subseteq (0,1)$ be countable dense sets that are invariant in both directions for $f$ and $g$ respectively.
Then there is a conjugacy $h$ of $f$ and $g$ satisfying $h(A)=B$.
\end{lemma}

\begin{definition}
For $f\in C(I)$ let us say that a point $z\in I$ is a \emph{left sharp local maximum} of $f$ if there is some $\delta>0$ such that $f(x)<f(z)$ for $x\in (z-\delta, z)$ and $f(x)\leq f(z)$ for $x\in (z, z+\delta)$. In a similar fashion we define \emph{left sharp local minimum}, \emph{right sharp local minimum} and \emph{right sharp local maximum}.
\end{definition}

\begin{notation}\label{conditions}
Let $M_f$ be the union of $\{0,1\}$ and the set of all left and right sharp local maxima and minima. It is easily shown that the set $M_f$ is countable.
For a closed set $F\subseteq I$ denote by $\Acc(F)$ the set of all accessible points of $F$ in $\mathbb R$, i.e. those points $x\in F$ for which there exists an open interval $(a,b)\subseteq \mathbb R\setminus F$ for which $x=a$ or $x=b$. 

For every $f\in C(I)$ let us denote by $C_f$ the smallest set such that
\begin{itemize}
\item[a)] $M_f\subseteq C_f$,
\item[b)] if $f^{-1}(y)$ contains an interval then $y\in C_f$,
\item[c)] if $n\in\mathbb N$ then $\Acc(\Fix(f^n))\subseteq C_f$,
\item[d)] $f(C_f)\subseteq C_f$,
\item[e)] if $y\in C_f$ then $\Acc(f^{-1}(y))\subseteq C_f$.
\end{itemize}
\end{notation}

\begin{lemma}
The set $C_f$ is countable for every $f\in C(I)$.
\end{lemma}

\begin{proof}
Let $S_1$ be the union of $M_f$, all the values of $f$ at locally constant points and all the sets $\Acc(\Fix(f^n))$ for $n\in\mathbb N$.
Clearly $S_1$ is countable. 
Let $S_{i+1}=S_i\cup f(S_i)\cup \bigcup\{\Acc(f^{-1}(y))\colon y \in S_i\}$. 
Clearly $C_f=\bigcup\{S_i\colon i\in\N\}$ and thus it is countable.
\end{proof}

Note that $C_f$ depends only on the topological properties of $I$ and the dynamics of $f$. That is if $f$ and $g$ are conjugate by some homeomorphism $h$, then $h(C_f)=C_g$. 
This is clear because $h$ maps $M_f$ onto $M_g$, locally constant intervals of $f$ to locally constant intervals of $g$ and periodic points of $f$ to periodic points of $g$.

Let us denote by $\mathcal J_f$ be the collection of all maximal open subintervals of $I\setminus C_f$. 

\begin{lemma}\label{complementaryintervals}
Let $J\in \mathcal J_f$. Then either $f\restriction_J$ is constant or $f\restriction_J$ is one to one and in this case $f(J)\in \mathcal J_f$. Also $f^{-1}(J)$ is the finite union (possibly the empty union) of elements of $\mathcal J_f$.
\end{lemma}

\begin{proof}
Let us prove first that $f\restriction_J$ is either constant or one-to-one. Suppose that the contrary holds. Then there are points $x, y, z\in J$ such that $f(x)=f(y)\neq f(z)$ and $x\neq y$. Let us suppose that $x<y<z$ and $f(x)<f(z)$ (the other possibilities are just easy modifications). Let $u=\min f\restriction_{[x,z]}$ and let $v=\max (f^{-1}(u)\cap [x, z])$. 
It follows that $v\in (x,z)$ is a right sharp local minimum. By Notation~\ref{conditions} a) it follows that $v\in C_f$ which is a contradiction since $J$ is disjoint from $C_f$.

Suppose now that $f\restriction_J$ is one-to-one and let us prove that $f(J)\in\mathcal J_f$. 
Observe first that $f(J)$ is disjoint from $C_f$, otherwise there would be a point $y\in f(J)\cap C_f$ and since
$f^{-1}(y)$ is a closed set not containing the whole set $J$ there will be a point in $\Acc f^{-1}(y)\cap J$ which is a contradiction with Notation~\ref{conditions} e).
We need to prove that $f(J)$ is a maximal interval disjoint from $C_f$. Suppose that $J=(a,b)$. Then there are $a_n, b_n\in C_f$ such that $a_n\to a, b_n\to b$. By continuity of $f$ it follows that $f(a_n)\to f(a)$ and $f(b_n)\to f(b)$. Also $f(a_n), f(b_n)\in C_f$ by Notation~\ref{conditions} d). Thus the maximality follows.

Observe first that $f^{-1}(J)$ is a countable union of disjoint collection of open intervals and if we prove that each of the intervals is mapped by $f$ onto $J$ it will follow by continuity that such a collection is in fact finite.
Denote $(a,b)=J$ and let $(c,d)$ be a maximal interval in $f^{-1}(J)$. Clearly $(c,d)\cap C_f=\emptyset$ by Notation~\ref{conditions} d), so it is enough to prove that it is maximal with this property. 
Note that $f(c), f(d)\in \{a, b\}$ otherwise we get a contradiction with $(c,d)$ being maximal interval in $f^{-1}(J)$. Also it can not happen that $f(c)=f(d)$ otherwise there will be a point of left local maximum or minimum in $(c,d)$ which would produce a point in $M_f\cap (c,d)$, which in turn would give a point in $C_f\cap J$, by Notation \ref{conditions} a), d). 
Hence $f((c,d))=J$. Moreover, by the first part of this proof we get that $f\restriction_{(c,d)}$ is one-to-one and thus it is either increasing or decreasing. Without loss of generality suppose the first case. Let us distinguish several possibilities.
If $f\geq f(c)$ on a left neighborhood of $c$ then $c$ is a point of right sharp local minimum and thus $c\in C_f$. Otherwise choose a sequence $a_n\in C_f$ such that $a_n\to a$. We define points $c_n=\max([0,c]\cap f^{-1}(a_n))$. These are eventually well defined, $c_n\to c$ and $c_n\in\Acc f^{-1}(a_n)$. Hence by Notation \ref{conditions} e) $c_n\in C_f$. We can proceed in a similar way with the point $d$ and thus the interval $(c,d)$ is maximal subinterval of $I\setminus C_f$.
\end{proof}

\begin{example}
For the tent map $f(x) =\min\{ 2x, 2(1-x)\}$, the set  $C_f$ contains all the dyadic numbers in $I$, 
thus $C_f$ is a dense subset of $I$ and hence $\mathcal J_f=\emptyset$.
For the map $g=\frac14 f$ we have 
\begin{align*}
C_g&=\{2^{-n}, 1-2^{-n}\colon n\in\mathbb N\}\cup\{0,1\}, \\
\mathcal J_g&=\{(2^{-n-1}, 2^{-n}), (1-2^{-n}, 1-2^{-n-1})\colon n\in\mathbb N\}.
\end{align*}
\end{example}

\begin{notation}
Let $G_f$ be a directed graph on $\mathcal J_f$ where $(J,K)$ forms an oriented edge if and only if $f(J)=K$. Note that for every $K\in \mathcal J_f$ there are only finitely many $J\in\mathcal J_f$ for which $f(J)=K$. Hence every vertex of the graph $G_f$ admits only finitely many arrows to enter.
Let $E_f=\mathbb Q\cap I\setminus \Cl(C_f)$ and let
\[D_f=\bigcup_{n\in\mathbb Z} f^n(E_f).\]
Note that the union is taken over all integers. In spite of that it follows by Lemma~\ref{complementaryintervals} that $D_f$ is countable.
Let us define 
\[\Psi(f)=(C_f\cup D_f, \leq\restriction_{C_f\cup D_f}, f\restriction_{C_f\cup D_f}).\]
\end{notation}

\begin{theorem}\label{conjugacy}
The mapping $\Psi$ is a reduction of orientation preserving conjugacy of interval dynamical systems to the isomorphism relation of countable structures.
\end{theorem}

\begin{proof}
Suppose first that $f$ is conjugate to $g$ via some increasing homeomorphism $h$, that is $f=h^{-1}gh$. We want to find an isomorphism $\varphi\colon \Psi(f) \to \Psi(g)$. Since $h$ does not need to map $D_f$ to $D_g$, we need to do some more work.
In fact we find a conjugacy $\bar h$ of $f$ and $g$ such that $\bar h(C_f\cup D_f)=C_g\cup D_g$. Then it will be enough to define a mapping $\varphi\colon C_f\cup D_f\to C_g\cup D_g$ as the restriction of $\bar h$. We will define $\bar h$ by parts.
First of all we define $\bar h$ on the set $\Cl(C_f)$ in the same way as $h$. 

Clearly $h$ induces an isomorphism of the graphs $(\mathcal J_f, G_f)$ and $(\mathcal J_g, G_g)$.
We will consider the components of the symmetrized graphs $G_f$ and $G_g$.
Note that $J, K\in\mathcal J_f$ are in the same component of $G_f$ if there are $m, n\geq 0$ such that $f^m(J)=f^n(K)$.

Let us distinguish two cases for the components of $G_f$. 
If a component of $G_f$ contains an oriented cycle, choose an element $J$ in there (note that the cycle is unique).
Hence there is $n\in\mathbb N$ such that $f^n(J)=J$.
By using Notation~\ref{conditions} c) it follows that either all the points of $J$ are fixed points for $f^n$ or there are no fixed points of $f^n$ in $J$ and the same has to be true for $g^n$ on $h(J)$. In the first case we just let $\bar h\restriction J$ to be any homeomorphism of $\Cl(J)$ and $\Cl(h(J))$, in the second case
we obtain by Lemma~\ref{folk} that there is a conjugacy $\bar h\restriction_{\Cl(J)}$ of $f^n\restriction_{\Cl(J)}$ and $g^n\restriction_{\Cl(h(J))}$ sending $D_f\cap J$ onto $D_g\cap h(J)$.
In components which do not contain an oriented cycle we choose $J$ arbitrarily, and let $\bar h\restriction_J$ be an arbitrary increasing homeomorphism $J\to h(J)$ which maps $J\cap D_f$ onto $h(J)\cap D_g$.

 For any $K$ that is in the same component as $J$ find 
 $m, n\geq 0$ such that $f^m(J)=f^n(K)\in\mathcal J_f$ and define $\bar h$ on $K$ using the definition of $\bar h$ on $J$ as 
\[(g^{-n}\restriction_{h(K)})g^m\bar h (f^{-m}\restriction_J)f^n.\]

On the other hand suppose that $\varphi$ is an isomorphism of the countable structure $\Psi(f)$ to $\Psi(g)$.
Hence $\varphi\colon C_f\cup D_f\to C_g\cup D_g$ is a bijection preserving the order. Thus it can be extended to an increasing homeomorphism $\tilde{\varphi}\colon I\to I$. We claim that $\tilde\varphi$ conjugates $f$ and $g$.
Consider any point $x\in C_f\cup D_f$ and compute  \[g(\tilde\varphi(x))=g(\varphi(x))=\varphi(f(x))=\tilde\varphi(f(x)),\]
where the middle equality follows form $\varphi$ being an isomorphism of $\Psi(f)$ and $\Psi(g)$.
Since the set $C_f\cup D_f$ is dense it follows by continuity that $g(\tilde\varphi(x))= \tilde\varphi(f(x))$ for every $x\in I$. Hence $f$ and $g$ are conjugate.
\end{proof}

\subsection{Borel coding}
We need to verify that the mapping $\Psi$ that was proved in Theorem \ref{conjugacy} to be a reduction
can be coded in a Borel way. We use standard notation for the Borel hierarchy, especially $\Sigma^0_1$ is used for the collection of all open sets, $\Sigma^0_2$ is used for the collection of  countable unions of closed sets etc. For a set $B\subseteq X\times Y$ and $x\in X$ let us denote by $B_x$ the set $\{y\in Y\colon (x,y)\in B\}$ and call it \emph{vertical section} of $B$.

The following seems to be folklore in descriptive set theory, but for the sake of completeness we include a 
proof.

\begin{proposition}\label{borelclosure}
Let $X, Y$ be Polish spaces and $B\subseteq X\times Y$ be a Borel set with countable vertical sections. Then the set $\bigcup_{x\in X} \{x\}\times \Cl(B_x)$ is Borel as well.
\end{proposition}

\begin{proof}
Let $\mathcal B$ be a countable base for the topology of $Y$. By the Lusin-Novikov selection theorem, we can assume that $B=\bigcup f_n$ for some Borel maps $f_n\colon X\to Y$.
It follows that 
\[(X\times Y)\setminus \left(\bigcup_{x\in X} \{x\}\times \Cl(B_x)\right)=\bigcup_{U\in \mathcal B}\bigcap_{n\in\mathbb N}((X\setminus f_n^{-1}(U))\times U).\]
Hence the set under discussion is Borel.
\end{proof}

Let us denote $\Gamma=\{(K, a)\in\mathcal K(I)\times I\colon a\in \Acc(K)\}$.

\begin{lemma}\label{AccPointsBorel}
The set $\Gamma$ is a $\Sigma^0_2$-set.
\end{lemma}

\begin{proof}
The sets
\[L_n=\{(K, a)\in \mathcal K(I)\times I\colon a\in K, K\cap (a-2^{-n}, a)=\emptyset\},\]
\[R_n=\{(K, a)\in \mathcal K(I)\times I\colon a\in K, K\cap (a, a+2^{-n})=\emptyset\}\]
are closed for every $n\in\mathbb N$. Hence the set $\bigcup (L_n\cup R_n)$ is a $\Sigma^0_2$-set.
\end{proof}

\begin{notation}
For a set $B\subseteq C(I)\times I$ let us define
\begin{align*}
B^\rightarrow&=\{(f, f(x))\colon (f, x)\in B\}, \\
B^\leftarrow&=\{(f, x)\colon (f,f(x))\in B\}, \\
B^\Leftarrow&=\{(f, x)\colon x\in Acc(f^{-1}(y)), (f,y)\in B\}.
\end{align*}
\end{notation}

\begin{lemma}\label{Borelimage}
Let $B\subseteq C(I)\times I$ be a Borel set with countable vertical sections. Then the sets
$B^\rightarrow, B^\leftarrow$ and $B^\Leftarrow$
are Borel as well.
\end{lemma}

\begin{proof}
The evaluation mapping $e\colon C(I)\times I\to I$, $e(f, x)=f(x)$ is continuous. Hence the mapping $\Phi\colon (f,x)\mapsto (f, e(f,x))$ is continuous as well. Especially, the restriction of $\Phi$ to $B$ is Borel and also countable-to-1.
Since by \cite[18.14]{Kechris} countable-to-1 image of a Borel set is Borel we conclude that $\Phi(B)=B^\rightarrow$ is Borel.

By the Lusin-Novikov selection theorem we can write $B=\bigcup F_n$ for some Borel maps $F_n$. It follows then that
\[B^\leftarrow=\bigcup_{n\in\mathbb N}\{(f, x)\colon e(f,x)=F_n(f)\}\] and thus it is a Borel set.

The mapping $p\colon C(I)\times I\to \mathcal K(I)$, $p(f, y)=f^{-1}(y)$ is upper semicontinuous and hence it is Borel by \cite[25.14]{Kechris}.
The set $\Gamma$ is Borel by Lemma~\ref{AccPointsBorel} and it has nonempty and countable vertical sections. Hence $\Gamma=\bigcup b_n$ for some Borel mappings $b_n\colon \mathcal K(I)\to I$, by the Lusin-Novikov selection theorem.
The mapping $\Psi_n\colon (f,y)\mapsto (f, b_n(f^{-1}(y)))=(f, b_n(p(f, y)))$ is a Borel mapping and its restriction to $B$ is countable-to-1. Hence by \cite[18.14]{Kechris} the set $\bigcup\Psi_n(B)=B^{\Leftarrow}$ is Borel.
\end{proof}               
%

\begin{lemma}\label{BorelComposition}
Let $X, Y, Z$ be standard Borel spaces, $f\colon X\to Y$ a Borel mapping and $R\subseteq Y\times Z$ a Borel binary relation. Then the set $R\circ f=\{(x, z)\colon (f(x), z)\in R\}$ is Borel.
\end{lemma}

\begin{proof}
Define $F\colon X\times Z\to Y\times Z$ by $F(x,z)=(f(x), z)$.
Clealy $F$ is a Borel mapping and $R=F^{-1}(R)$ which is consequently a Borel set.
\end{proof}

\begin{lemma}\label{Borel}
The set 
$$
A=\{(f, x)\in C(I)\times I\colon x\in C_f\cup D_f\}
$$ 
is a Borel subset of $C(I)\times I$.
\end{lemma}

\begin{proof}
Let us prove first that the set $B_a:=\{(f,x)\colon x\in M_f\}$ is Borel.
As the set $\{(f, x)\colon x \text{ is a left sharp local maximum}\}$ can be written in the form
\[\bigcup_{\varepsilon>0}\bigcap_{\eta>0}\bigcup_{\delta>0}\{(f,x)\colon \forall z\in [x-\varepsilon, x-\eta]: f(z)\leq f(x)-\delta \And \forall z\in [x,x+\varepsilon]: f(z)\leq f(x)\}\]
it follows that it is a $\Sigma^0_3$ set. By symmetry it follows that $B_a$ is the union of four $\Sigma^0_3$-sets and thus it is Borel.
The set $B_b:=\{(f, y)\colon f^{-1}(y) \text{ contains an interval} \}$ is a $\Sigma^0_2$-set. 
Let $B_c:=\{(f, x)\in C(I)\times I\colon x\in\Acc(\Fix(f^n)), n\in\N\}$.
The mapping $F_n\colon C(I)\to \mathcal K(I)$, $F_n(f)=\Fix(f^n)$ is upper semicontinuous
(since if $f_i$ converge uniformly to $f$ and $x_i$ converge to $x$ with $f_i^n(x_i)=x_i$ then $f^n(x)=\lim_i f_i^n(\lim_j x_j)=\lim_j\lim_i f_i^n(x_j)=\lim_j f^n(x_j)=f^n(x)$ by the Moore-Osgood theorem) and thus it is Borel. Since $\Gamma$ is a Borel set by Lemma~\ref{AccPointsBorel} we conclude that the composition $\Gamma\circ F_n$ is Borel because the composition of a Borel binary relation and a Borel mapping (in that order) is a Borel relation by Lemma \ref{BorelComposition}. Hence $B_c=\bigcup_{n\in\mathbb N} \Gamma\circ F_n$ is Borel.
Hence the set $B=B_a\cup B_b\cup B_c$ is Borel.
Define recursively $B_1=B$, $B_{n+1}=B_n \cup B_n^\rightarrow \cup B_n^\Leftarrow$ for $n\in\mathbb N$. All these sets are Borel by Lemma~\ref{Borelimage}.
It follows that $A_1=\bigcup B_n=\{(f, x)\colon x\in C_f\}$ is Borel.

Since $A_1$ has countable vertical sections and it is Borel we conclude using Proposition~\ref{borelclosure} that $A_2=\bigcup_{f\in C(I)}(\{f\}\times \Cl(A_{1,f})) $ is Borel as well. Consequently $A_3=\{(f, x)\colon x\in E_f\}= (C(I)\times \mathbb Q)\setminus A_2$ is Borel. By Lemma~\ref{Borelimage} we conclude that all the sets $A_{n+1}=A_n\cup A_n^\rightarrow\cup A_n^\leftarrow$, $n\geq 3$ are Borel.
Finally $A=A_1\cup \bigcup_{n\geq 3} A_n$ is a Borel set.
\end{proof}

\begin{theorem}\label{bireducible}
The orientation preserving conjugacy of interval dynamical systems is Borel bireducible to the $S_\infty$-universal orbit equivalence relation.
\end{theorem}

\begin{proof}
By the result of \cite[Section 4.2]{Hjorth} orientation preserving conjugacy of increasing interval homeomorphisms is Borel bireducible to the $S_\infty$-universal orbit equivalence relation. Hence especially the $S_\infty$-universal orbit equivalence relation is Borel reducible to increasing conjugacy of orientation preserving interval dynamical systems.

Let us argue for the converse.
The set $A$ from Lemma~\ref{Borel} is Borel and it has nonempty and countable vertical sections. Hence by the Lusin-Novikov selection theorem we can find Borel mappings $F_n\colon C(I)\to I$ such that $\bigcup F_n=A$.
Since all the vertical sections are infinite we can additionally suppose that for every pair $(f,x)\in A$ there is exactly one $n\in\mathbb N$ 
satisfying $F_n(f)=x$. Let
\[\Phi(f)=(\mathbb N, R, m),\] 
where $R$ is a binary relation and $m$ is a unary function
such that $aRb$ iff $F_a(f)\leq F_b(f)$ and $m(a)=b$ iff $f(F_a(f))=F_b(f)$ for $a, b\in\mathbb N$. There is a natural isomorphism $\Phi(f)\to\Psi(f)$, $a\mapsto F_a(f)$. 
Hence clearly $\Phi$ is a reduction.
It is routine to verify that $\Phi$ is Borel by the fact that the mappings $F_n$ are Borel.
\end{proof}

Let us note that the same conclusion as in the previous theorem can be proved without assuming orientation preserving conjugacy but with just conjugacy. The reason is that in the proofs of Theorem \ref{bireducible} and Theorem \ref{conjugacy} we can simply consider a ternary betweenness relation $T$ instead of the binary relation of linear order $\leq$, i.e. $(x,y,z)\in T$ if and only if $y$ is an element of the smallest interval containing $x$ and $z$. This ternary relation is clearly forgetting the order of $I$. Also by \cite[Exercise 4.14]{Hjorth} $E_{S_\infty}$ is Borel reducible to conjugacy of interval homeomorphisms. Thus we get the following result.

\begin{theorem}\label{bireducible2}
The conjugacy of interval dynamical systems is Borel bireducible to the $S_\infty$-universal orbit equivalence relation.
\end{theorem}

We note that Theorem~\ref{bireducible2} is a special case of 
Hjorth's conjecture \cite[Conjecture 10.6]{Hjorth} stating that every equivalence relation induced by a continuous 
action of the group $\mathcal H(I)$ of all interval homeomorphisms on a Polish space is classifiable by countable structures. In this case the homeomorphism group acts on the space of continuous selfmaps by conjugacy. Similarly one can prove that the orbit equivalence relations induced by natural left or right composition actions of the homeomorphism group on the space of continuous selfmaps is Borel reducible to the $S_\infty$-universal equivalence relation.
Also it is known that the orbit equivalence induced by the homeomorphism group action $\mathcal H(I)$ on the hyperspace $\mathcal K(I)$ is Borel bireducible to the $S_\infty$-universal orbit equivalence relation (see \cite[Exercise~4.13]{Hjorth} or \cite{ChangGao2019} for a proof). All these are special cases of Hjorth's conjecture.

\section{Hilbert cube dynamical systems}
Since the homeomorphism equivalence relation of metrizable compacta is known to be Borel bireducible to the universal orbit equivalence relation,
it is not surprising that conjugacy of dynamical systems on the Hilbert cube is of the same complexity, which is the main result of this section.
In a dynamical system $(X, f)$, a point $x$ is called a \emph{locally attracting fixed point} if $f(x)=x$ and there is a neighborhood $U$ of $x$ such that for every $z\in U$ the trajectory $(f^n(z))_{n\in\mathbb N}$ converges to $x$.
The notion of a $Z$-set in the Hilbert cube $Q$ plays an important role and it describes a kind of relative homotopical smallness. 

\begin{definition}
A closed subset of a (separable metric) space $X$ is called a \emph{$Z$-set} in $X$ if for every open cover $\mathcal U$ of $X$ and every continuous function $f$ of the Hilbert cube $Q$ into $X$ there is a continuous function $g\colon Q\to X$ such that $f$ and $g$ are $\mathcal U$-close (i.e. for every $x\in X$ there is $U\in\mathcal U$ such that $f(x), g(x)\in U$) and $g(Q)\cap A=\emptyset$.
\end{definition}
An introduction to this notion can be found in \cite[Chapter 5]{vanMill}.
Mostly, we will need the following properties on $Z$-sets in the Hilbert cube. First, every homeomorphism of $Z$-sets can be extended to a homeomorphism of the Hilbert cube \cite[Theorem 5.3.7]{vanMill}. Second, the Hilbert cube $Q\times I$ contains a topological copy of itself $Q\times \{0\}$ as a $Z$-set \cite[Lemma 5.1.3]{vanMill} and similarly the base in the cone of the Hilbert cube is a Z-set. Third, every closed subset of a $Z$-set in $Q$ is a $Z$-set in $Q$ \cite[Lemma 5.1.2]{vanMill}.
If follows from the first and second property that there is topologically just one way, how to embed the Hilbert cube into itself as a Z-set (namely $Q\times\{0\}$ included in $Q\times I$).
For the purpose of this paper, an \emph{absolute retract} is just a space homeomorphic to a retract of the Hilbert cube (which is equivalent to being a retract of every separable metric space, in which it is embedded).
A space $X$ is said to have the \emph{disjoint cell property} if for every $\varepsilon>0$, $n\in\mathbb N$ and continuous mappings $f, g:I^n\to X$ there are continuous mappings $f', g': I^n\to X$ with disjoint images such that $f$ and $f'$ as well as $g$ and $g'$ are $\varepsilon$-close. The last two notions give a topological characterization of the Hilbert cube.

\begin{theorem}[{Toru\'nczyk, see e.g. \cite[Theorem 4.2.25]{vanMill}}]
A space $X$ is homeomorphic to the Hilbert cube if and only if it is an absolute retract with the disjoint cell property.
\end{theorem}

The following proposition is a special case of \cite[Proposition~1]{Zielinski} and it can be easily proved using the back and forth argument. Another reference for the proof is \cite[Proposition 9]{Lorch}.
Our formulation is using a slightly different language. 

\begin{proposition}\label{extension}
Let $K\subseteq A, L\subseteq B$ be four nonempty compact metrizable spaces such that $A\setminus K$ and $B\setminus L$ are dense sets of isolated points in $A$ and $B$ respectively. Then every homeomorphism of $K$ onto $L$ can be extended to a homeomorphism of $A$ onto $B$.
\end{proposition}

The following will be useful in the proof of Theorem~\ref{Hilbert}.

\begin{proposition}[{\cite[Theorem~2.6]{GladdinesvanMill} or \cite[Corollary 4.2.24]{vanMill}}]\label{homotopy}
If $X$ is a nondegenerate Peano continuum then there exists a homotopy $H\colon  \mathcal K(X)\times I\to \mathcal K(X)$ for which
\begin{itemize}[noitemsep]
\item $H(A, 0)=A$ for every $A\in 2^X$,
\item $H(A, t)$ is finite for every $t>0$ and $A\in 2^X$.
\end{itemize}
\end{proposition}

Recall that if $Y\subseteq X$ and $\varepsilon>0$ we say that $X$ is $\varepsilon$-deformable into $Y$ if there exists a continuous mapping $\varphi\colon X\times I\to X$ such that $\varphi(x,0)=x$, $\varphi(x,1)\in Y$ and the diameter of $\varphi(\{x\}\times I)$ is at most $\varepsilon$ for every $x\in X$.
The following proposition was proved in \cite[1.1 and 1.3]{Krasinkiewicz}.

\begin{proposition}\label{AR}
Let $X$ be a compact space such that for every $\varepsilon>0$ there exists an absolute retract $Y\subseteq X$ for which $X$ is $\varepsilon$-deformable into $Y$. Then $X$ is an absolute retract.
\end{proposition}

By a result of \cite{Anderson} the union of two  Hilbert cubes, whose intersection is a Z-set in each of the cubes and which is homeomorphic to the Hilbert cube, is the Hilbert cube again. By the result of \cite{Handel}, even a weaker condition is enough to get the same conclusion:

\begin{proposition}\label{HilbertUnion}
Let $X$ be a space which is the union of two Hilbert cubes $Q_1$ and $Q_2$.
Suppose that $Q_1\cap Q_2$ is a Hilbert cube which is a Z-set in $Q_1$. Then $X$ is a Hilbert cube.
\end{proposition}

It should be noted here that a space which is the union of two Hilbert cubes intersecting in a Hilbert cube may not be a Hilbert cube \cite{Sher}.

\begin{lemma}\label{HilbertInfiniteUnion}
Let $X$ be a compact metric space which is the union of Hilbert cubes $Q$, $Q_1$, $Q_2, \dots$ such that $Q_i\cap Q_j=\emptyset$ and $Q\cap Q_i$ is a Z-set in $Q_i$ for every $i, j\in\mathbb N$, $i\neq j$. Suppose moreover that the diameter of $Q_i$ tends to zero. Then $X$ is homeomorphic to the Hilbert cube as well.
\end{lemma}

\begin{proof}
Let us denote $X_i=Q\cup Q_1\cup \dots\cup Q_i$ and observe that it is homeomorphic to the Hilbert cube for every $i\in\mathbb N$ by an inductive usage of Proposition \ref{HilbertUnion}.
To make the same conclusion for $X$ we use the Toru\'nczyk's theorem.

There is topological just one way, how to embed the Hilbert cube into itself as a Z-set.
Hence every pair $(Q_i, Q_i\cap Q)$ is equivalent to $(Q\times I, Q\times\{0\})$ 
because $Q_i\cap Q$ is a Z-set in $Q_i$ and it is homeomorphic to the Hilbert cube. 
Thus simply there is a homotopy $h_i\colon Q_i\times I\to Q_i$ such that $h_i(x,t)=x$ for $t=0$ or $x\in Q_i\cap Q$ and $h_i(x,1)\in Q\cap Q_i$.
Let us denote 
\[
s_i(x, t)=
\begin{cases} 
 x, & x\in X_i, \\ h_j(x, t),  &x\in Q_j,\quad j>i.    
\end{cases}
\]
Since the diameter of $Q_i$ tends to zero, the diameters of $s_i(\{x\}\times I)$ are sufficiently small for large $i$. Hence for every $\varepsilon>0$ it follows that $X$ is $\varepsilon$-deformable into $X_i$ for some $i\in\mathbb N$.
Moreover, for every $i\in\mathbb N$, $X_i$ is an absolute retract.
Hence $X$ is an absolute retract by Proposition \ref{AR}. 

Let us argue that $X$ has the disjoint cell property (see \cite[p. 294]{vanMill}).
Denote $r_i(x)=s_i(x, 1)$. Then $r_i\colon X\to X_i$ is a retraction. 
Suppose that $f, g\colon  I^n\to X$ are continuous mappings and $\varepsilon>0$. Then for sufficiently large $i\in\mathbb N$ diameters of $Q_j$ are smaller than $\varepsilon$ for $j\geq i$. Hence $r_i$ is $\varepsilon$-close to identity on $X$. Since $X_i$ is homeomorphic to the Hilbert cube it has the disjoint cell property and thus there are continuous mappings $f', g'\colon  I^n\to X_i$ with disjoint images such that $r_i\circ f$ and $f'$ as well as $r_i\circ g$ and $g'$ are $\varepsilon$-close. 
It follows that by the triangle inequality that $f$ and $f'$ as well as $g$ and $g'$ are $2\varepsilon$-close.
Thus $X$ has the disjoint cell property.
As mentioned at the beginning of the proof, by Toru\'nczyk's theorem it follows that $X$ is homeomorphic to the Hilbert cube.
\end{proof}

An equivalence relation $E$ on a Borel subset $Y$ of a Polish space $X$ is said to be \emph{countably separated} if there is a sequence $(Z_n)_{n=1}^{\infty}$ of $E$-invariant Borel subsets of $Y$ such that for all $x,y \in Y$, the points $x$ and $y$ are $E$-equivalent if and only if the sets $\{ n \in \N \, ; \ x \in Z_n \}$ and $\{ n \in \N \, ; \ y \in Z_n \}$ are equal.
A \emph{transversal} for an equivalence relation $E\subseteq X\times X$ is a set $T\subseteq X$ whose intersection with every $E$-equivalence class is a one point set.

The following proposition by Burgess is a special kind of a selection theorem and it will serve as a useful tool to complete the Borel coding argument, which is by no means straightforward.

\begin{proposition}[\cite{Burgess}] \label{Burgess}
Let $G$ be a Polish group, $X$ a Polish space and let $\alpha$ be a continuous action of $G$ on $X$. Denote by $E$ the orbit equivalence relation induced by $\alpha$ and let $Y$ be an $E$-invariant Borel subset of $X$. Let $E_Y$ be the restriction of $E$ to $Y$ and assume that $E_Y$ is countably separated. Then there is a Borel transversal for $E_Y$.
\end{proposition}

In the next proposition, we denote by $\mathcal K_X(Q)$ the collection of subspaces of $Q$ which are homeomorphic to $X$. It is known for a long time that this is always a Borel set \cite{Ryll-Nardzewski}. 

\begin{proposition}\label{HilbertSelection}
There is a Borel mapping $\gamma\colon  \mathcal K_Q(Q) \to \Inj(Q, Q)$ such that the image of $\gamma(R)$ equals to $R$.
\end{proposition}

\begin{proof}
Let $G$ be the homeomorphism group of $Q$ and let us consider the action $\alpha$ of G on $\Inj(Q, Q)$ given by $g\cdot h= h\circ g$. It follows that the corresponding orbit equivalence relation $E$ induced by $\alpha$ satisfies that $f E g$ if and only if the images of $f$ and $g$ are equal. Moreover $E$ is countably separated as if we consider a countable base $\mathcal B$ of $Q$ and $Z_B=\{f\in \Inj(Q, Q)\colon  Im(f)\cap B\neq\emptyset\}$ for $B\in\mathcal B$ then $f E g$ if and only if $\{B\in\mathcal B\colon   f\in Z_B\}=\{B\in\mathcal B\colon  g\in Z_B\}$ and also the sets $Z_B$ are clearly invariant with respect to $E$.
By a straightforward application of Proposition \ref{Burgess} we get that there is a Borel transversal $T$ of $E$. As the mapping $\chi\colon  f\mapsto Im(f)$, $\Inj(Q, Q)\to \mathcal K(Q)$ is Borel (even continuous), the graph of $\chi$ is a Borel subset of $\Inj(Q, Q)\times \mathcal K(Q)$. As moreover $T$ is a Borel subset of the domain of $\chi$ and $\chi$ is one-to-one on $T$ it follows that the mapping $\gamma=(\chi|T)^{-1}$ has a Borel graph and thus it is a Borel mapping. Clearly for every $R\in\mathcal K_Q(Q)$ we get that $\gamma(R)$ is an embedding of $Q$ into $Q$ whose image equals $R$.
\end{proof}

Some ideas for the proof of the following comes from the paper \cite{KrupskiVejnar}.

\begin{theorem}\label{Hilbert} 
The conjugacy of Hilbert cube homeomorphisms (or selfmaps) is Borel bireducible to $E_{G_\infty}$.
\end{theorem}

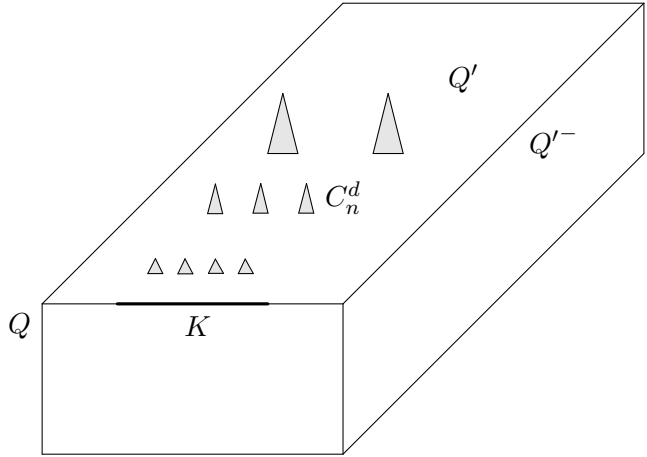
\begin{figure}
\begin{tikzpicture}[line cap=round,line join=round,>=triangle 45,x=2.0cm,y=2.0cm]
\clip(-0.8870230685673408,-1.0495897586754248) rectangle (6.689259282742075,2.8063219870557705);
\fill[line width=0.4pt,fill=black,fill opacity=0.10000000149011612] (1.5,1.) -- (1.7,1.) -- (1.6,1.4) -- cycle;
\fill[line width=0.4pt,fill=black,fill opacity=0.10000000149011612] (2.2,1.) -- (2.4,1.) -- (2.3,1.4) -- cycle;
\fill[line width=0.4pt,fill=black,fill opacity=0.1] (1.1,0.6) -- (1.2,0.6) -- (1.1516705507080174,0.7973955765263389) -- cycle;
\fill[line width=0.4pt,fill=black,fill opacity=0.1] (1.4015873351160641,0.6) -- (1.501587335116064,0.6034837273061294) -- (1.4532578858240814,0.8008793038324683) -- cycle;
\fill[line width=0.4pt,fill=black,fill opacity=0.1] (1.7055408494552977,0.6034837273061294) -- (1.8055408494552976,0.6034837273061294) -- (1.757211400163315,0.8008793038324683) -- cycle;
\fill[line width=0.4pt,fill=black,fill opacity=0.1] (0.702161987805911,0.20130445208116723) -- (0.8021619878059107,0.20130445208116723) -- (0.752959293246011,0.30131164716636394) -- cycle;
\fill[line width=0.4pt,fill=black,fill opacity=0.1] (0.9021763779763046,0.19971703628616413) -- (1.0021763779763042,0.19971703628616413) -- (0.9497988518263981,0.2997242313713608) -- cycle;
\fill[line width=0.4pt,fill=black,fill opacity=0.1] (1.1053655997367042,0.20289186787617036) -- (1.205365599736704,0.20289186787617036) -- (1.1514006577917946,0.2981368155763577) -- cycle;
\fill[line width=0.4pt,fill=black,fill opacity=0.1] (1.3037925741120944,0.20130445208116723) -- (1.4037925741120942,0.20130445208116723) -- (1.349827632167185,0.2981368155763577) -- cycle;
\draw [line width=0.4pt] (0.,0.)-- (2.,0.);
\draw [line width=0.4pt] (0.,-1.)-- (2.,-1.);
\draw [line width=0.4pt] (0.,0.)-- (0.,-1.);
\draw [line width=0.4pt] (2.,0.)-- (2.,-1.);
\draw [line width=0.4pt] (2.,2.)-- (4.,2.);
\draw [line width=0.4pt] (2.,2.)-- (0.,0.);
\draw [line width=0.4pt] (4.,2.)-- (2.,0.);
\draw [line width=0.4pt] (4.,2.)-- (4.,1.);
\draw [line width=0.4pt] (4.,1.)-- (2.,-1);
\draw [line width=1.2pt] (0.5,0.)-- (1.5,0.);
\draw [line width=0.4pt] (1.5,1.)-- (1.7,1.);
\draw [line width=0.4pt] (1.7,1.)-- (1.6,1.4);
\draw [line width=0.4pt] (1.6,1.4)-- (1.5,1.);
\draw [line width=0.4pt] (2.2,1.)-- (2.4,1.);
\draw [line width=0.4pt] (2.4,1.)-- (2.3,1.4);
\draw [line width=0.4pt] (2.3,1.4)-- (2.2,1.);
\draw [line width=0.4pt] (1.1,0.6)-- (1.2,0.6);
\draw [line width=0.4pt] (1.2,0.6)-- (1.1516705507080174,0.7973955765263389);
\draw [line width=0.4pt] (1.1516705507080174,0.7973955765263389)-- (1.1,0.6);
\draw [line width=0.4pt] (1.4015873351160641,0.6034837273061294)-- (1.501587335116064,0.6034837273061294);
\draw [line width=0.4pt] (1.501587335116064,0.6034837273061294)-- (1.4532578858240814,0.8008793038324683);
\draw [line width=0.4pt] (1.4532578858240814,0.8008793038324683)-- (1.4015873351160641,0.6034837273061294);
\draw [line width=0.4pt] (1.7055408494552977,0.6034837273061294)-- (1.8055408494552976,0.6034837273061294);
\draw [line width=0.4pt] (1.8055408494552976,0.6034837273061294)-- (1.757211400163315,0.8008793038324683);
\draw [line width=0.4pt] (1.757211400163315,0.8008793038324683)-- (1.7055408494552977,0.6034837273061294);
\draw [line width=0.4pt] (0.702161987805911,0.20130445208116723)-- (0.8021619878059107,0.20130445208116723);
\draw [line width=0.4pt] (0.8021619878059107,0.20130445208116723)-- (0.752959293246011,0.30131164716636394);
\draw [line width=0.4pt] (0.752959293246011,0.30131164716636394)-- (0.702161987805911,0.20130445208116723);
\draw [line width=0.4pt] (0.9021763779763046,0.19971703628616413)-- (1.0021763779763042,0.19971703628616413);
\draw [line width=0.4pt] (1.0021763779763042,0.19971703628616413)-- (0.9497988518263981,0.2997242313713608);
\draw [line width=0.4pt] (0.9497988518263981,0.2997242313713608)-- (0.9021763779763046,0.19971703628616413);
\draw [line width=0.4pt] (1.1053655997367042,0.20289186787617036)-- (1.205365599736704,0.20289186787617036);
\draw [line width=0.4pt] (1.205365599736704,0.20289186787617036)-- (1.1514006577917946,0.2981368155763577);
\draw [line width=0.4pt] (1.1514006577917946,0.2981368155763577)-- (1.1053655997367042,0.20289186787617036);
\draw [line width=0.4pt] (1.3037925741120944,0.20130445208116723)-- (1.4037925741120942,0.20130445208116723);
\draw [line width=0.4pt] (1.4037925741120942,0.20130445208116723)-- (1.349827632167185,0.2981368155763577);
\draw [line width=0.4pt] (1.349827632167185,0.2981368155763577)-- (1.3037925741120944,0.20130445208116723);
\draw[color=black] (-0.15,-0.15) node {$Q$};
\draw[color=black] (2.8,1.5) node {$Q{'}$};
\draw[color=black] (3.4,1.0750956992997227) node {$Q{'}^-$};
\draw[color=black] (1.0339220556696547,-0.15) node {$K$};
\draw[color=black] (2,0.7077712308941867) node {$C^d_n$};
\end{tikzpicture}
\caption{The compactum $Q_K$}
\end{figure}

\begin{proof}
For one direction it is enough to prove that the homeomorphism equivalence relation of metrizable compacta is Borel reducible to conjugacy of Hilbert cube homeomorphisms because the first relation is Borel bireducible to $E_{G_\infty}$ by the main result of \cite{Zielinski}.
To this end let 
\begin{align*}
Q&=\{x\in\ell_2\colon 0\leq x_n\leq 1/n\}, \\
Q'&=Q\times I\times \{0\}, \\
Q''&=Q\times I\times [-1,1],\\
Q'^-&=Q\times I\times [-1,0]
\end{align*}
and let $\|\cdot\|$ be the usual norm on $\ell_2$.

Let us fix a homotopy $H\colon \mathcal K(Q)\times I\to \mathcal K(Q)$ given by Proposition~\ref{homotopy} for the case $X=Q$.
Let us fix $K\in \mathcal K(Q)$. We want to find a homeomorphism $f_K$ of a Hilbert cube $Q_K\subseteq Q''$ such that the topological information about $K$ is somehow encoded in the dynamics of $f_K$. Let $D_n^K=H(K, 2^{-n})$, $n\in\mathbb N$. 
Let $\varepsilon_n$ be the minimum of $1/n$ and the smallest distance of different points in $D_n^K$. For every $d\in D_n^K$ 
fix a set 
\[
B^d_n=\{(x, 2^{-n},0)\in Q''\colon \|d- x\|\leq\varepsilon_n/3\}.
\]
It follows that $B^d_n$ is always homeomorphic to the Hilbert cube since it is affinely homeomorphic to an infinite dimensional compact convex subset of a Hilbert space \cite{Keller}.
Let $C^d_n$ be the cone in $Q''$ with base $B^d_n$ and with the vertex $(d, 2^{-n}, 2^{-n})$, $d\in D_n^K$, $n\in\mathbb N$, 
i.e., the union of all segments with end points $(d, 2^{-n}, 2^{-n})$ and $p$, $p\in B^d_n$.
The cone over the Hilbert cube is homeomorphic to the Hilbert cube \cite[Theorem 1.7.5]{vanMill}, which applies to $C^d_n$.
Let $Q_K^m=Q'^-\cup\bigcup\{ C^d_n\colon n\in\mathbb N, n\leq m, d\in D_n^K\}$ and $Q_K=\bigcup\{Q_K^m\colon m\in\mathbb N\}$ (see Figure 1).

Since $Q'^-\cap C^d_n=B^d_n$ is homeomorphic to the Hilbert cube, which is a Z-set in $C^d_n$, we inductively obtain by Lemma \ref{HilbertInfiniteUnion} that $Q_K$ is a Hilbert cube.

Let $h(x)=\sqrt{x}, x\in I$ or any fixed homeomorphism of $I$ with two fixed points $0,1$; and $1$ 
being a locally attracting fixed point. We define
\[f_K(x)=\begin{cases}x, &x\in Q'^-,
\\ 
((1-h(t))a+h(t)d, 2^{-n}, 2^{-n} h(t)), &
\begin{matrix}
x=((1-t)a+td, 2^{-n}, 2^{-n}t)\in C^d_n,\\
d\in D_n^K,t\in I, n\in\N.
\end{matrix}
\end{cases}\]


All the points in $Q'^-$ are fixed points for $f_K$ and these are clearly not attracting.
All the points in $\bigcup D_n^K$ are fixed points of $f_K$ and these are attracting. There are no other fixed points of $f_K$.
It follows that $K$ is homeomorphic (or even equal) to the set of fixed points that are limits of attracting points but not attracting by itself (and thus defined only by dynamical notions). Hence if $f_K$ and $f_L$ are conjugate then $K$ and $L$ are homeomorphic, $K, L\in \mathcal K(Q)$.

On the other hand if $K, L$ are homeomorphic compacta in $Q$ then the sets $K\cup \bigcup_{n\in\mathbb N}D_n^K\times \{2^{-n}\}$ and $L\cup\bigcup_{n\in\mathbb N}D_n^L\times \{2^{-n}\}$ are homeomorphic by Lemma~\ref{extension}. This homeomorphism can be simply extended to a homeomorphism
\[\varphi\colon (K\times \{(0,0)\})\cup \bigcup\{B^d_n(K)\colon d\in D_n^K, n\in\mathbb N\} \to (L\times\{(0, 0)\})\cup \bigcup \{B^d_n(L)\colon d\in D_n^L, n\in\mathbb N\}.\]
Both the sets in the domain and range of $\varphi$ are $Z$-sets in $Q'^-$ since these are closed subsets of the $Z$-set $Q'\times\{0\}$ \cite[Lemma 5.1.2, Lemma 5.1.3]{vanMill}. Hence $\varphi$ can be extended to a homeomorphism $\varphi'\colon Q'^-\to Q'^-$ \cite[Theorem 5.3.7]{vanMill}.
It remains to extend $\varphi'$ 
linearly on the cones to obtain a homeomorphism $\varphi''$.
It follows that $\varphi''$ conjugates $f_K$ and $f_L$.
Note that we can identify $f_K$ with its graph and thus it can be considered as a closed subspace of $Q''\times Q''$.
To verify that the mapping $\chi\colon  \mathcal K(Q)\to \mathcal K(Q''\times Q'')$, $K\mapsto f_K$ is Borel is a routine which is usually omitted in this type of proofs.


However, we are still not done, since $f_K$ is defined on the topological copy of the Hilbert cube $Q_K$ which differs when changing $K$. 
Let us consider the Borel mapping $\gamma$ given by Proposition \ref{HilbertSelection}.
We redefine the mapping $f_K$ by conjugating it via $\gamma(Q_K)$ in the following way.
The mapping $K\mapsto \gamma(Q_K)^{-1} \circ f_K\circ (\gamma(Q_K))$, $\mathcal K(Q)\to \mathcal H(Q)$ is the desired Borel reduction.


To conclude the proof it is enough to Borel reduce conjugacy of Hilbert cube maps to $E_{G_\infty}$. Consider structures of the form $(Q, R)$ where $R$ is a closed binary relation on $Q$. Two such structures $(Q, R)$ and $(Q, S)$ are said to be isomorphic if there is a homeomorphisms $\psi\colon Q\to Q$ for which $(\psi\times\psi)(R)=S$.
By a fairly more general result \cite{RosendalZielinski} it follows that such isomorphism equivalence relation is Borel reducible to $E_{G_\infty}$.
There is a Borel (even continuous) reduction which takes a continuous map $f\colon Q\to Q$ and assigns $(Q, \text{graph}(f))$ to it. Combining the two reductions we get the desired one.
\end{proof}




\section{Concluding remarks and questions}

Let us summarize some of the results on the complexity of conjugacy equivalence relation in Table~1 in which we consider conjugacy equivalence relation of 
maps, homeomorphisms, and pointed transitive homeomorphisms of the 
interval,  circle,  Cantor set and Hilbert cube, respectively.
Let us recall that a \emph{pointed dynamical system} is a triple $(X, f, x)$, where $(X, f)$ is a dynamical system and $x\in X$. We say, that a pointed dynamical system $(X, f, x)$ is \emph{transitive} if the forward orbit of $x$ in $(X, f)$ is dense. 
Two pointed dynamical systems $(X, f, x)$ and $(Y, g, y)$ are called \emph{conjugate} if there is a conjugacy of $(X,f)$ and $(Y, g)$ mapping $x$ to $y$.
We proceed by a series of simple notes as comments on Table~\ref{table}.

\begin{table}[h]
\centerline{
\begin{tabular}{ r  l  l }
\hline
& Homeomorphisms/maps & Pointed transitive homeomorphisms 
\\
\hline
Interval  & $E_{S_\infty}$ \cite{Hjorth}, Theorem \ref{bireducible2}  & $\emptyset$ Note \ref{arc}     \\
Circle   & $E_{S_\infty}$ Note \ref{circle} & $E_=$ Note \ref{circlepointed}  \\
Cantor set  & $E_{S_\infty}$ \cite{CamerloGao} & $E_{=^+}$ \cite{Kaya}, Note \ref{Cantor}  \\
Hilbert cube  & $E_{G_\infty}$ Theorem \ref{Hilbert} & ? Question \ref{question} \\
\hline
\end{tabular}
}
\caption{The complexity of conjugacy equivalence relation.}
\label{table}
\end{table} 

\begin{note}
Conjugacy of pointed transitive maps of the interval is smooth; indeed it is enough to assign to every pointed transitive dynamical system $(I, f, x)$ the $\N\times\N$ matrix of true and false: $(f^m(x)<f^n(x))_{m, n\in\mathbb N}$ which determines $f$ uniquely up to increasing conjugacy.
\end{note}

\begin{note}\label{arc}
There are no transitive homeomorphisms on the interval.
\end{note}

\begin{note}\label{circle}
The complexity result by Hjorth \cite[Section 4.2]{Hjorth} that conjugacy of interval homeomorphisms is Borel bireducible to $E_{S_\infty}$, remains true for circle homeomorphisms simply by a modification of the original proof. A modification of the proof of Theorem \ref{bireducible2} will give a similar result for circle maps. The same method as for the interval case can be used just by considering left or right local maxima and minima defined in an obvious way and then iterating this set forward and backward in a similar manner as in Notation \ref{conditions}.
Thus conjugacy of circle selfmaps is Borel bireducible to the $S_\infty$-universal orbit equivalence relation.
\end{note}

\begin{note}\label{circlepointed}
Transitive homeomorphisms of the circle are well known to be conjugate to irrational rotations. Hence the rotation number is a complete invariant and hence conjugacy of (pointed) transitive homeomorphisms of the circle is Borel bireducible to the equality on irrationals (or on an uncountable Polish space).
\end{note}

\begin{note}\label{Cantor}

By a result of Kaya \cite{Kaya}, conjugacy of pointed minimal homeomorphisms of the Cantor set is Borel bireducible to the equality of countable sets $E_{=^+}$. Note that his proof works in the same vein for pointed transitive homeomorphisms of the Cantor set.
Let us recall the main part of his construction in this case. Let $X$ be the Cantor set and $\mathcal B$ the collection of all clopen sets in $X$.  To a pointed transitive system $(X, f, x)$ we assign the collection
\[\mathsf{Ret}(f, x)=\{\mathsf{Ret_B}(f, x)\colon B\in\mathcal B\},\]
where $\mathsf{Ret_B}(f, x)=\{n\in\Z\colon f^n(x)\in B\}$.
It can be verified that the mapping $\Phi$ defined as
\[\Phi(f, x)=(\mathsf{Ret_B}(f, x)\colon B\in\mathcal B)\in \mathcal P(\Z)^{\mathcal B}\]
is a reduction of pointed transitive Cantor maps to the equality of countable sets in $\mathcal P(\Z)^{\mathcal B}$, 
i.e., $(f,x)$ is conjugate to $(g, y)$ if and only if $\mathsf{Ret}(f,x)=\mathsf{Ret}(g,y)$.
\end{note}

The following question is the missing part to complete Table~\ref{table}.

\begin{question}\label{question}
What is the complexity of conjugacy of transitive pointed Hilbert cube homeomorphisms (or maps)?
\end{question}

It was explained to us by Burak Kaya, that conjugacy equivalence relation of pointed transitive Hilbert cube homeomorphisms is a Borel relation \cite{KayaAmbit}.
The main reason is that every conjugacy of such systems preserves the distinguished point and thus it is automatically prescribed on a dense subset. Hence there is at most one conjugacy between such systems. Let us note that neither $E_{S_\infty}$ nor $E_{G_\infty}$ is Borel and thus these equivalence relations can not answer Question \ref{question}.


Since triangular maps i.e., maps $f\colon I^2\to I^2$ of the form $f(x,y)=(g(x), h(x,y))$ for continuous 
maps $g\colon I\to I$ and $h\colon I^2\to I$, 
lie in between one-dimensional and two dimensional and there is a gap in the complexity of the last 
two mentioned equivalence relations, the following question is natural.

\begin{question}
What is the complexity of conjugacy of triangular maps?
Is it Borel bireducible to $E_{S_\infty}$ or to $E_{G_\infty}$?
\end{question}

Positive answer to the next question would provide a strengthening of Theorem~\ref{bireducible2}.

\begin{question}
Is conjugacy of closed binary relations on the closed interval Borel reducible to the $S_\infty$-universal orbit equivalence relation?
\end{question}

The answer to the preceding question is affirmative if Hjorth's conjecture is true.

\bibliographystyle{alpha}
\bibliography{citace}

\begin{thebibliography}{{Kay}17c}

\bibitem[And67]{Anderson}
Richard~D. Anderson.
\newblock Topological properties of the {H}ilbert cube and the infinite product
  of open intervals.
\newblock {\em Trans. Amer. Math. Soc.}, 126:200--216, 1967.

\bibitem[BS95]{BuescuStewart}
Jorge Buescu and Ian Stewart.
\newblock Liapunov stability and adding machines.
\newblock {\em Ergodic Theory Dynam. Systems}, 15(2):271--290, 1995.

\bibitem[Bur79]{Burgess}
John~P. Burgess.
\newblock A selection theorem for group actions.
\newblock {\em Pacific J. Math.}, 80(2):333--336, 1979.

\bibitem[CG01]{CamerloGao}
Riccardo Camerlo and Su~Gao.
\newblock The completeness of the isomorphism relation for countable {B}oolean
  algebras.
\newblock {\em Trans. Amer. Math. Soc.}, 353(2):491--518, 2001.

\bibitem[CG19]{ChangGao2019}
Cheng Chang and Su~Gao.
\newblock The complexity of the classification problems of finite-dimensional
  continua.
\newblock {\em Topology Appl.}, 267:106876, 18, 2019.

\bibitem[Cle09]{Clemens}
John~D. Clemens.
\newblock Isomorphism of subshifts is a universal countable {B}orel equivalence
  relation.
\newblock {\em Israel J. Math.}, 170:113--123, 2009.

\bibitem[For18]{ForemanIntroduction}
Matthew Foreman.
\newblock What is a {B}orel reduction?
\newblock {\em Notices Amer. Math. Soc.}, 65(10):1263--1268, 2018.

\bibitem[Gao09]{Gao}
Su~Gao.
\newblock {\em Invariant descriptive set theory}, volume 293 of {\em Pure and
  Applied Mathematics (Boca Raton)}.
\newblock CRC Press, Boca Raton, FL, 2009.

\bibitem[GJS16]{GaoJacksonSeward}
Su~Gao, Steve Jackson, and Brandon Seward.
\newblock Group colorings and {B}ernoulli subflows.
\newblock {\em Mem. Amer. Math. Soc.}, 241(1141):vi+241, 2016.

\bibitem[GvM93]{GladdinesvanMill}
Helma Gladdines and Jan van Mill.
\newblock Hyperspaces of {P}eano continua of {E}uclidean spaces.
\newblock {\em Fund. Math.}, 142(2):173--188, 1993.

\bibitem[Han78]{Handel}
Michael Handel.
\newblock On certain sums of {H}ilbert cubes.
\newblock {\em General Topology and Appl.}, 9(1):19--28, 1978.

\bibitem[Hjo00]{Hjorth}
Greg Hjorth.
\newblock {\em Classification and orbit equivalence relations}, volume~75 of
  {\em Mathematical Surveys and Monographs}.
\newblock American Mathematical Society, Providence, RI, 2000.

\bibitem[Kay17a]{KayaSubshift}
Burak Kaya.
\newblock The complexity of the topological conjugacy problem for {T}oeplitz
  subshifts.
\newblock {\em Israel J. Math.}, 220(2):873--897, 2017.

\bibitem[Kay17b]{Kaya}
Burak Kaya.
\newblock The complexity of topological conjugacy of pointed {C}antor minimal
  systems.
\newblock {\em Arch. Math. Logic}, 56(3-4):215--235, 2017.

\bibitem[{Kay}17c]{KayaAmbit}
Burak {Kaya}.
\newblock {On the complexity of topological conjugacy of compact metrizable
  $G$-ambits}.
\newblock {\em arXiv e-prints}, page arXiv:1706.09821, June 2017.

\bibitem[Kec95]{Kechris}
Alexander~S. Kechris.
\newblock {\em Classical descriptive set theory}, volume 156 of {\em Graduate
  Texts in Mathematics}.
\newblock Springer-Verlag, New York, 1995.

\bibitem[Kel31]{Keller}
Ott-Heinrich Keller.
\newblock Die {H}omoiomorphie der kompakten konvexen {M}engen im {H}ilbertschen
  {R}aum.
\newblock {\em Math. Ann.}, 105(1):748--758, 1931.

\bibitem[Kra76]{Krasinkiewicz}
J\'ozef Krasinkiewicz.
\newblock On a method of constructing {ANR}-sets. {A}n application of inverse
  limits.
\newblock {\em Fund. Math.}, 92(2):95--112, 1976.

\bibitem[KV20]{KrupskiVejnar}
Pavel Krupski and Benjamin Vejnar.
\newblock The complexity of the homeomorphism relations on some classes of
  compacta.
\newblock {\em J. Symb. Log.}, pages 1--19, 2020.

\bibitem[Lor81]{Lorch}
Edgar~R. Lorch.
\newblock On some properties of the metric subalgebras of {$l^{\infty }$}.
\newblock {\em Integral Equations Operator Theory}, 4(3):422--434, 1981.

\bibitem[Mel07]{Melleray}
Julien Melleray.
\newblock Computing the complexity of the relation of isometry between
  separable {B}anach spaces.
\newblock {\em MLQ Math. Log. Q.}, 53(2):128--131, 2007.

\bibitem[MR17]{MottoRos}
Luca Motto~Ros.
\newblock Can we classify complete metric spaces up to isometry?
\newblock {\em Boll. Unione Mat. Ital.}, 10(3):369--410, 2017.

\bibitem[RN65]{Ryll-Nardzewski}
C.~Ryll-Nardzewski.
\newblock On a {F}reedman's problem.
\newblock {\em Fund. Math.}, 57:273--274, 1965.

\bibitem[RZ18]{RosendalZielinski}
Christian Rosendal and Joseph Zielinski.
\newblock Compact metrizable structures and classification problems.
\newblock {\em J. Symb. Log.}, 83(1):165--186, 2018.

\bibitem[She77]{Sher}
R.~B. Sher.
\newblock The union of two {H}ilbert cubes meeting in a {H}ilbert cube need not
  be a {H}ilbert cube.
\newblock {\em Proc. Amer. Math. Soc.}, 63(1):150--152, 1977.

\bibitem[ST17]{SabokTsankov}
Marcin Sabok and Todor Tsankov.
\newblock On the complexity of topological conjugacy of {T}oeplitz subshifts.
\newblock {\em Israel J. Math.}, 220(2):583--603, 2017.

\bibitem[Tho13]{Thomas}
Simon Thomas.
\newblock Topological full groups of minimal subshifts and just-infinite
  groups.
\newblock In {\em Proceedings of the 12th {A}sian {L}ogic {C}onference}, pages
  298--313. World Sci. Publ., Hackensack, NJ, 2013.

\bibitem[vM01]{vanMill}
Jan van Mill.
\newblock {\em The infinite-dimensional topology of function spaces}, volume~64
  of {\em North-Holland Mathematical Library}.
\newblock North-Holland Publishing Co., Amsterdam, 2001.

\bibitem[Zie16]{Zielinski}
Joseph Zielinski.
\newblock The complexity of the homeomorphism relation between compact metric
  spaces.
\newblock {\em Adv. Math.}, 291:635--645, 2016.

\end{thebibliography}
\end{document}